\documentclass[a4paper,11pt,reqno]{article}
\usepackage{amsmath,amssymb,hyperref,amsthm,enumerate,color,stmaryrd,pgf,
tikz,comment,multirow}
\usetikzlibrary{positioning,arrows.meta,bending}
\usepackage[normalem]{ulem}
\usepackage{marginnote}
\usepackage{authblk}
\usepackage[left=2.6cm,right=2.6cm,top=3cm,bottom=3cm,bindingoffset=0cm]{geometry}
\usepackage{float}
\title{\bf On nowhere-zero $\boldsymbol{3}$-flows in Cayley graphs on groups of order $\boldsymbol{p^2n}$} 
\author[1,2]{Milad Ahanjideh\thanks{This work was supported in part by the Slovenian Research and Innovation Agency (ARIS), research projects N1-0391 and J1-70035.}}
\author[1,2]{Istv\'an Kov\'acs\thanks{This work was supported in part by the Slovenian Research and Innovation Agency (ARIS), research program P1-0285 and reserach projects 
N1-0391, J1-50000, N1-0428, N1-0429, J1-70035, J1-70047, J1-70046 and 
N1-0481.}} 
\affil[1]{FAMNIT, University of Primorska, Galgolja\v{s}ka 8, SI-6000 Koper, Slovenia}
\affil[2]{IAM, University of Primorska, Muzejski trg 2, SI-6000 Koper, Slovenia}
\newtheorem{thm}{Theorem}[section]
\newtheorem{lem}[thm]{Lemma}
\newtheorem{prop}[thm]{Proposition}

\newtheorem{conj}[thm]{Conjecture}

\theoremstyle{remark}

\def\aut{\mathrm{Aut}}

\def\cay{\mathrm{Cay}}
\def\edge{\mathrm{E}}
\def\G{\Gamma}
\def\AGaL{\mathrm{A\Gamma L}}
\def\AGL{\mathrm{AGL}}

\def\Null{\mathrm{Null}}

\def\vertex{\mathrm{V}}
\def\Z{\mathbb{Z}}
\def\F{\mathbb{F}}
\newcommand{\sg}[1]{\langle {#1}\rangle}

\begin{document}
\maketitle
\begin{abstract}
We consider Tutte's $3$-flow conjecture in the class of Cayley graphs on groups $G$ of order $p^2n$, where $p$ is a prime and $n$ is square-free and coprime to $p$. 
The conjecture has been proved by J. Zhang and Z. Zhang in the case where $n$ is a prime (Discrete Mathematics, 2023); and by the authors in the case where $p=2$ (arXiv:2603.24175). In this paper, we prove the conjecture in the case where $p > 2$ and either a Sylow $p$-subgroup or a Hall $n$-subgroup of $G$ is a cyclic group. 
\\ [+1ex]

[+1ex]
\noindent \emph{Math. Subj. Class.:} 05C21, 05C25.
\end{abstract}
\section{Introduction}\label{sec:intro}

Let $\G$ be a graph without loops and parallel edges allowed. 
The vertex set of $\G$ is denoted by $\vertex(\G)$ and the edge set by $\edge(\G)$. An \emph{orientation} of $\G$ is a digraph $D$ obtained from $\G$ by assigning one of the two possible directions to each edge of $\G$. 
Note that parallel edges may be assigned opposite directions. For an edge $e \in \edge(\G)$, we denote by $D(e)$ the associated arc of $D$.
For a vertex $v \in \vertex(\G)$, we denote by $D^+(v)$ the set of edges of $\G$ 
that the associated arcs in $D$ have tail $v$ and by $D^-(v)$ the set of those edges 
that the associated arcs in $D$ have head $v$. Let $A$ be an abelian group written additively and with identity element $0$. An \emph{$A$-flow} in $\G$ is a pair  $(D,\varphi)$, 
where $D$ is an orientation of $\G$ and $\varphi$ is an $A$-valued function on the edge 
set $\edge(\G)$, such that for every vertex $v \in  \vertex(\G)$, 
\[
\sum_{e\in D^{+}(v)}\varphi(e)=\sum_{e\in D^{-}(v)}\varphi(e).
\]
Moreover, if $\varphi(e) \ne 0$ for every $e \in \edge(\G)$, then $(D,\varphi)$ is called a \emph{nowhere-zero $A$-flow}. For a fixed positive integer $k$, a $\Z$-flow $(D,\varphi)$ is called a $k$-flow if 
$|\varphi(e)|< k$ for every  $e \in \edge(\G)$, where $\Z$ is the additive group 
of integers.
 
 The concept of a flow is due to Tutte~\cite{T49,T54}, who introduced it as a generalization of map coloring. Tutte showed that the existence of $\Z_k$-flow in a graph $\G$ is equivalent to the existence of $k$-flow (see, e.\,g.~\cite[Theorem~21.3]{BMbook}). He proposed three conjectures on integer flows in graphs, all of which are still open. 
Our main interest in this paper is his $3$-flow conjecture. 

\begin{conj}\label{conj}
(Tutte's $3$-flow conjecture) 
Every $4$-edge-connected graph admits a nowhere-zero $3$-flow. 
\end{conj}

This conjecture has been studied extensively. Jaeger~\cite{J} proved that every $4$-edge-connected graph admits a nowhere-zero $4$-flow. He also conjectured that there exists
a positive integer $k$ such that every $k$-edge-connected graph admits a 
nowhere-zero $3$-flow. Thomassen~\cite{T} confirmed Jaeger's conjecture for $k=8$, and this result was later improved by Lov\'asz et al.\,\cite{LTWZ}, who proved that 
every $6$-edge-connected graph admits a nowhere-zero $3$-flow. 

In this paper, our focus is on the validity of Conjecture~\ref{conj} for Cayley graphs. 
It is a well-known fact that every connected 
$k$-valent Cayley graph is $k$-edge-connected~\cite{M}. Therefore, in the class of Cayley graphs, Conjecture~\ref{conj} is equivalent to saying that every connected Cayley graph on a group $G$ of valency at least $4$ admits a nowhere-zero $3$-flow. This has been proved to be true for 
abelian groups~\cite{PSS}, nilpotent groups~\cite{NS}, dihedral groups~\cite{FL}, generalized dihedral groups~\cite{LL}, generalized quaternion groups~\cite{LL}, 
generalized dicyclic groups~\cite{AI}, supersolvable groups with a non-cyclic Sylow $2$-subgroup~\cite{ZsZ}, and for groups with a derived subgroup of square-free order~\cite{ZsZ}. It has been also proved for groups of 
order $p^2q$~\cite{ZZ} and $8p$~\cite{ZhZ}, where $p, q$ are distinct primes, and for solvable groups of order $2n$~\cite{AK}, where $n$ is square-free. 
The main result of this paper is the following theorem, which generalizes the case of groups of order $p^2q$~\cite{ZZ}: 

\begin{thm}\label{main}
Let $p$ be an odd prime and $n$ be a positive square-free integer 
coprime to $p$. Then every Cayley graph of valency at least four on a group of order $p^2n$ with either a cyclic Sylow $p$-subgroup or a cyclic Hall $n$-subgroup admits a nowhere-zero $3$-flow.
\end{thm}

Note that every finite group whose order is not divisible by $4$ is solvable. 
Hence, the Hall $n$-subgroup in the above theorem always exists. 

\medskip
The paper is organized as follows. In the next section we collect the preliminary results used in our proof. We present the proof of Theorem~\ref{main} in Section~\ref{sec:proof}.
\section{Preliminaries}\label{sec:known}

All groups and graphs in this paper are finite and unless otherwise stated, all graphs are simple. 
For a subset $E \subseteq \edge(\G)$, we denote by $\G - E$ the subgraph  obtained from $\G$ by deleting all edges in $E$.  
If each vertex of $\G$ has even valency, then $\G$ is called \emph{even}. A subgraph 
$\G'$ of $\G$ is called a \emph{parity subgraph} if $\G-\edge(\G')$ is an 
even graph. 

\begin{prop}[{\cite[Lemma~2.4]{ZZ}}]\label{parity}
Let $\G$ be a graph and $k$ be a positive integer greater than $2$. If $\G$ has a parity 
subgraph admitting a nowhere-zero $k$-flow, then $\G$ admits a nowhere-zero 
$k$-flow. 
\end{prop}

Let $G$ be a group with identity element $1_G$. A multiset $X \subseteq G$ is called a \emph{connection multiset} if $1_G \notin X$, $X=X^{-1}=\{x^{-1} \mid x \in X\}$, and for every $x \in X$, the multiplicity of $x$ and $x^{-1}$, respectively, are equal. 
Then the \emph{Cayley multigraph} $\cay(G,X)$ is defined to have vertex set $G$ and for any $g, h \in G$, there are exactly $\ell$ edges joining $g$ and $h$ if and only if $g^{-1}h \in X$ with multiplicity $\ell$ in $X$. 
In the case where $X$ is a usual subset of $G$, 
$\cay(G,X)$ is a simple graph, i.\,e., it does not posses parallel edges 
and we refer to it as a \emph{Cayley graph}.

\medskip
Let $\G=\cay(G,X)$ be a Cayley multigraph on a group $G$. For $x \in X$, we refer to the edges of the form $\{g,gx\}$, where $g$ runs over $G$, as the \emph{$x$-edges}. 
Note that the set of $x$-edges of $\G$ coincides with the set of its $x^{-1}$-edges.     
Suppose that $N$ is a normal subgroup of $G$ with $X \cap N=\emptyset$. 
Then the multiset $X/N \subseteq G/N$, defined as $X/N=\{ Nx \mid x \in X\}$, 
is a connection multiset of the quotient group $G/N$. The Cayley multigraph $\cay(G/N,X/N)$ is 
called the \emph{quotient} of $\G$ with respect to $N$, denoted by $\G/N$. 

\begin{prop}[{\rm \cite[Proposition~4.1]{NS}}]\label{quo}
Let $G$ be a group and $N$ be a normal subgroup of $G$. 
Let $X$ be a connection multiset of $G$ such that $N \cap X=\emptyset$. 
If $\cay(G/N,X/N)$ admits a nowhere-zero $k$-flow, then so does $\cay(G,X)$.
\end{prop}

Let $n \ge 2$ be an integer. The \emph{circular ladder} $CL_n$ is the Cayley graph 
$\cay(\Z_n \times \Z_2,\{(1,0),(n-1,0),(0,1)\})$, where $\Z_n$ denotes the 
additive group of integers modulo $n$.  The $(0,1)$-edges of $CL_n$ are called the \emph{rungs} and the remaining edges are called the \emph{rail edges}.  The \emph{M\"obius ladder} $ML_n$ is the Cayley graph 
$\cay(\Z_{2n},\{1,2n-1,n\})$. The $n$-edges of $ML_n$ are called the \emph{rungs} 
and the remaining edges are called the \emph{rail edges}. 
By a \emph{closed ladder} we mean a circular or M\"obius ladder.  
The use of closed ladders in the study of Tutte's conjecture for Cayley graphs was initiated in~\cite{NS}, where the following result was established.

\begin{prop}[{\cite[Theorem~3.3]{NS}}]\label{c-involution}
Let $\cay(G,X)$ be a Cayley multigraph of valency at least $4$ such that 
$X$ contains a central involution of $G$. Then $\cay(G,X)$ admits a nowhere-zero 
$3$-flow. 
\end{prop}

The method of closed ladders was later developed by Zhang and Zhou~\cite{ZsZ}, who proved the following 
results. 

\begin{prop}[{\cite[Lemma~4.8]{ZsZ}}]\label{twoladders}
Let $\cay(G, X)$ be a Cayley multigraph of valency at least five. Suppose that $X$ contains two inverse-closed subsets $\{u_1, u_2\}, \{v_1,v_2\}$ and an involution $x$ 
such that $\cay(\langle u_1,u_2, x \rangle,\{u_1, u_2, x\})$ and 
$\cay(\langle v_1, v_2, x \rangle,\{v_1, v_2, x\})$ are both closed ladders with the $x$-edges 
being their rungs. Then $\cay(G, X)$ admits a nowhere-zero $3$-flow. 
\end{prop}

\begin{prop}[{\cite[Theorem~1.3]{ZsZ}}]\label{derived}
Every Cayley multigraph of valency at least $4$ on a group whose derived subgroup is of square-free order admits a nowhere-zero $3$-flow.
\end{prop}


We follow~\cite{Hbook} for the terminology and notation used in group theory. If $x, y \in G$, then $|x|$ denotes 
the \emph{order} of $x$, and $x^y$ is written for the \emph{conjugate} $y^{-1}xy$. 
The \emph{semidirect product} of $G$ with another group $H$ is denoted by $G \rtimes H$. 
The \emph{center}, \emph{derived subgroup} and \emph{Fitting subgroup} of $G$ are denoted by $Z(G), G^\prime$ and $F(G)$, respectively.  
For a subgroup $H \le G$, the \emph{centralizer} of 
$H$ in $G$ is denoted by $C_G(H)$. We shall use the following facts.

\begin{prop}[{\cite[Theorem~1.4.5]{Hbook}}]\label{N/C}
If $N$ is a normal subgroup of $G$, then $G/C_G(N)$ is isomorphic to a subgroup of $\aut(N)$.
\end{prop}

\begin{prop}[{\cite[Theorem~3.4.2(b)]{Hbook}}]\label{Fit1}
If $G$ is a solvable group, then $C_G(F(G)) \le F(G)$. 
\end{prop}
\begin{prop}[{\cite[Theorem~2.3.12 ]{Hbook}}]\label{primitive}
Let $G$ be a primitive permutation group on a set $\Omega$ with the following properties:
\begin{enumerate}[(i)]
\item $G$ contains a minimal normal subgroup $N$ that is abelian. 
\item For $\alpha \in \Omega$, $G_\alpha$ contains contains an abelian normal subgroup $A$ such that $N$ is  a minimal normal subgroup of $NA$. 
\end{enumerate}
Then $|\Omega|=p^n$ for a prime $p$ and 
$G$ is permutation isomorphic to a subgroup of $\AGaL_1(p^n)$.
\end{prop}

For the next proposition, which provides a sufficient condition for a $5$-valent simple graph to admit a nowhere-zero $\mathbb{Z}_3$-flow, we need some additional notation.
For a function $\varphi \colon \edge(\G) \to \Z_3$, let 
\[
\Null(\varphi)=\{ e \in \edge(\G) \mid \varphi(e)=0\}.
\] 
Finally, for an integer $k \geq 1$, we use the symbol $[k]$ to denote the set $\{1,\ldots,k\}$.

\begin{prop}[{\cite[Lemma~2.9]{AK}}]\label{new-suffi}
Let $\G$ be a simple graph of order $n$ and valency $5$ and suppose that it contains  subgraphs $\Sigma$ and $\Lambda$ satisfying the following conditions.
\begin{enumerate}[(i)]
\item $\Sigma$ is a spanning subgraph of $\G$ with components 
$\Sigma_1,\ldots,\Sigma_k$, and  there exists an odd integer $t > 1$ such that 
$\Sigma_i \cong CL_t$ for each $i \in [k]$ (hence $n=2tk$).
\item $\Sigma \cup \Lambda$ is a parity subgraph of $\G$.  
\item For each $i \in [k]$, 
$|\edge(\Lambda) \cap \edge(\Sigma_i)| \in \{1,2\}$ and 
every edge in $\edge(\Lambda) \cap \edge(\Sigma_i)$ is a rung of $\Sigma_i$.
\item $\Lambda$ admits a $\Z_3$-flow $(D,\varphi)$ such that 
\begin{enumerate}[(a)]
\item $\Null(\varphi) \subseteq \edge(\Lambda) \cap \edge(\Sigma)$. 
\item For each $i \in [k]$, $\edge(\Lambda) \cap \edge(\Sigma_i) \not\subset \Null(\varphi)$.
\end{enumerate}
\end{enumerate}
Then $\G$ admits a nowhere-zero $\Z_3$-flow. 
\end{prop}


\begin{prop}[{\cite[Lemma~4.3]{AK}}]\label{multi}
Let $G$ be a finite group and $X$ be a connection multiset of $G$ of cardinality $5$ (computed with multiplicity) such that $X$ contains an element of order larger than $2$ and of multiplicity $1$. Then $\cay(G,X)$ admits a nowhere-zero $\Z_3$-flow. 
\end{prop}

We conclude this section with a lemma about $5$-valent Cayley graphs on groups which have a subgroup of index $2$.

\begin{lem}\label{condition}
Let $G$ be a group of order $2n$ and $H$ be a subgroup of $G$ of order $n$.
If $\G=\cay(G,X)$ is a connected $5$-valent Cayley graph with $|X \cap H|=2$, 
then $\G$ admits a nowhere-zero $3$-flow.
\end{lem}
\begin{proof}
Let $a \in G \setminus H$. As $H$ has index $2$ in $G$, $G$ decomposes as 
$G=H \cup Ha$.
It is easy to check that $\Sigma:=\cay(G,X \setminus H)$ is  
a bipartite subgraph of $\G$ of valency $3$ with parts $H$ and $Ha$. 
This graph admits a nowhere-zero $3$-flow by a well known result saying that a 
$3$-valent graph admits a nowhere-zero 3-flow if and only if it is bipartite (see~\cite[Proposition~6.4.2]{Dbook}). Since $\G-\edge(\Sigma)$ is an even graph, $\Sigma$ is also a parity subgraph of $\G$; applying Proposition~\ref{parity} completes the proof.
\end{proof}
\section{Proof of Theorem~\ref{main}}\label{sec:proof}

We argue by contradiction and assume that $\G=\cay(G,X)$ is a counterexample of the smallest order. Recall that $G$ is a group of order 
$p^2n$, where $p$ is an odd prime and $n$ is square-free and coprime to $p$. 
Moreover, either a Sylow $p$-subgroup or a Hall $n$-subgroup of $G$ is cyclic. By assumption, $\G$ is a connected graph of valency at least 
$4$, which does not admit a nowhere-zero $\Z_3$-flow. 
We will derive a contradiction through a sequence of lemmas. 

\begin{lem}\label{L1}
The valency of $\G$ is $5$.
\end{lem}
\begin{proof}
If $\G$ has valency $4$, then $\G$ is Eulerian and hence admits a 
nowhere-zero $\Z_3$-flow in a straightforward manner. 
If the valency is larger than $5$, then $\G$ is 
$6$-edge-connected by~\cite{M} and 
therefore, it admits a nowhere-zero $\Z_3$-flow by~\cite{LTWZ}. 
Either case contradicts our assumption that $\G$ does not admit a 
nowhere-zero $\Z_3$-flow. 
\end{proof}

\begin{lem}\label{L2}
The Fitting subgroup $F(G)$ is equal to the Sylow $p$-subgroup of $G$. 
\end{lem}
\begin{proof}
Let $P$ be a Sylow $p$-subgroup of $G$. 
Assume for the moment that $|F(G)|$ is square-free.  
Then $F(G)$ is a cyclic group and it follows from Proposition~\ref{Fit1} 
that $C_G(F(G))=F(G)$. This implies that $G/F(G)$ is isomorphic to a 
subgroup of $\aut(F(G))$ by Proposition~\ref{N/C}. As the later subgroup is abelian, 
$G^\prime \le F(G)$, in particular, $G^\prime$ has square-free order. 
However, then $\G$ admits a nowhere-zero $3$-flow due to Proposition~\ref{derived}, a contradiction. 
We obtain $P \leq F(G)$, in particular, $P$ is normal in $G$. 

If $X \cap P=\emptyset$, then the Cayley (multi)graph $\cay(G/P,X/P)$ 
admits a nowhere-zero $3$-flow by Proposition~\ref{derived}, and therefore, so does $\G$ by 
Proposition~\ref{quo}. This is impossible, hence there exist 
$a, a^{-1} \in X \cap P$. 

Assume that $P < F(G)$. We finish the proof by showing that this leads to a contradiction.
Then $O_q(G) \ne 1$ for some prime divisor $q$ of $n$.  
For ease of notation, let $Q=O_q(G)$.
Then $Q \cong \Z_q$.  If $q=2$, then $Q \le Z(G)$. This implies 
that $X$ contains a central involution, and 
so $\G$ admits a nowhere-zero $3$-flow by Proposition~\ref{c-involution}. 
This is excluded, so $q > 2$. 

Assume for the moment that $X \cap Q \neq \emptyset$. 
Then 
\[
X=\{a,a^{-1},b,b^{-1},x\},
\]
where $b \in Q$, and $x$ is an involution. 
Observe that $G=(P \times Q) \rtimes \Z_2$. 

We first deal with the case where $P \cong \Z_{p^2}$. 
Then, $P=\sg{a}$ and $a^x=a$ or $a^x=a^{-1}$. 
The graph $\cay(\sg{a,x},\{a,a^{-1},x\})$ is  
circular ladder. On the other hand, 
$b^x=b$ or $b^x=b^{-1}$, and 
$\cay(\sg{b,x},\{b,b^{-1},x\})$ is also a circular ladder.
Proposition~\ref{twoladders} can be applied to $\G$ with 
$\{u_1,u_2\}=\{a,a^{-1}\}$ and $\{v_1,v_2\}=\{b,b^{-1}\}$, and we obtain that 
$\G$ admits a nowhere-zero $3$-flow, a contradiction.

We turn second to the case where $P \cong \Z_p \times \Z_p$. 
Then $a$ and $a^x$ must generate $P$. We obtain $aa^x \ne 1$ and is fixed 
by $x$. Thus, $P_1:=\sg{aa^x}$ is a normal subgroup of $G$ of order $p$. 
Clearly, $P_1 \cap X=\emptyset$. By Proposition~\ref{derived}, 
$\cay(G/P_1,X/P_1)$ admits a nowhere-zero $3$-flow, and hence 
by Proposition~\ref{quo}, so does $\G$, a contradiction. 
We conclude that $X \cap Q=\emptyset$. 

It follows from the minimality of $\G$ and Proposition~\ref{quo} that 
$X/Q$ is a multiset. By Proposition~\ref{multi}, the element $Qa \in X/Q$ 
has multiplicity $2$, hence there exists $b \in Q$ such that 
$ba, (ba)^{-1} \in X$. We deduce that 
$G=(P \times Q) \rtimes \sg{x}$. If $P\cong \Z_{p^2}$, then $P=\sg{a}$. 
Define the subgroup $P_1$ of $P$ as 
$P_1=\sg{a^p}$ if $P \cong \Z_{p^2}$, and $P_1=\sg{aa^x}$ if 
$P \cong \Z_p\times\Z_p$. It follows that $P_1 \lhd G$ and 
$P_1 \cap X=\emptyset$. Applying Propositions~\ref{quo} and 
\ref{derived} to $\cay(G/P_1,X/P_1)$, we obtain that $\G$ admits a nowhere-zero $3$-flow, a contradiction. 
\end{proof}

\begin{lem}\label{L3}
$F(G) \cong \Z_p \times \Z_p$. 
\end{lem}
\begin{proof}
For ease of notation, let $F=F(G)$. 
By Lemma~\ref{L1}, $|F|=p^2$, hence we have 
$G=F \rtimes H$, where $H$ is a Hall $n$-subgroup of $G$.

Assume to the contrary that $F \cong \Z_{p^2}$. 
By Propositions~\ref{N/C} and~\ref{Fit1}, $H$ is isomorphic to a subgroup of 
$\aut(F)$. As $F \cong \Z_{p^2}$, $\aut(F) \cong \Z_{p^2-p}$, and so 
$H \cong \Z_n$, where $n$ is an even divisor of $p-1$. 
Note that $G$ is a Frobenius group with Frobenius kernel $F$ and 
Frobenius complement $H$. 
For $h \in H$, $h \ne 1_G$, it is easy to show that  
$Fh=\{x^{-1}hx \mid x \in F\}$. This gives $|xh|=|h|$ holds for every 
$x \in F$ and $h\in H, h \ne 1_G$. 

Let $x$ be an involution contained in $X$. Note that $\sg{F,x}$ is a dihedral group of order $2p^2$, which contains all involutions of $G$.   
Let $F_1$ denote the subgroup of $F$ of order $p$. 
As $F_1 \lhd G$, $|F_1 \cap X| \ge 2$ follows from 
Propositions~\ref{quo} and~\ref{derived} and the assumption that 
no nowhere-zero $3$-flow exists in $\G$. 

If $|F \cap X| >2$, then $G=\sg{X}=\sg{F,x}$.
As $\sg{F,x}$ is isomorphic to a dihedral group of order $2p^2$, $\G$ admits a nowhere-zero $3$-flow by~\cite{FL}. 
Thus we may assume that 
\[
F \cap X=F_1 \cap X~\text{and}~|X \setminus F|=3.
\]

If $X \setminus F$ contains only involutions, then 
$\sg{X}=\sg{F,x}$, and we are done again. Therefore, we may write that 
$X \setminus F=\{b,b^{-1},x\}$ for some $b \notin F$ and $|b| > 2$. 
Then $G/F=\sg{Fb,Fx} \cong \Z_{n}$. Let $n=2m$. Note that $m$ is odd and 
$m > 1$. 
Also, $|Fb|=|b|$ and $|Fx|=2$, and we deduce that 
$|b|=m$ or $2m$. Now, it follows from Lemma~\ref{condition} that 
$|b|=m$. 

We may write that $H=\sg{b,c}$, $|c|=2$, and 
\begin{equation}\label{eq:conj}
b^{-1}xb=x^r,~c^{-1}xc=x^{-1},\quad x\in F,
\end{equation}
where $r$ has multiplicative order $m$ modulo $p^2$.

Since $\sg{X}=G$, it follows that the unique involution $x$ in $X$ is of the form $x=ac$, where $F=\sg{a}$, and $X$ can be expressed as 
\[
X=\{a^s,a^{-s},b,b^{-1},x\},~s \in [p^2-1],~s\equiv 0\!\!\!\!\pmod p.
\]

We finish the proof of the lemma by defining two subgraphs $\Sigma$ and $\Lambda$ of $\G$, and then 
verifying that these satisfy conditions (i)--(iv) in Proposition~\ref{new-suffi}.

Define $\Sigma$ as 
\[
\Sigma=\cay(G,\{a^s,a^{-s},x\}).
\]

The definition of $\Lambda$ requires some preparation. 
Let $L=\sg{bx}$. Then $|L|=|bx|=|bac|=|bc|=n$. Define the 
$2n$-cycle $\Lambda_0$ as 
\[
\Lambda_0 = 1 - b - bx - bxb - (bx)^2 - \cdots - (bx)^{n-1} - (bx)^{n-1}b - 1. 
\]
For $0 \le j \le n-1$, define the $m$-cycle $C_j$ as 
\[
C_{j}= (bx)^j - (bx)^jb - (bx)^jb^2 - (bx)^jb^3 - \cdots - (bx)^jb^{m-1} - (bx)^j.
\]
Then glue all cycles $C_j$ to $\Lambda_0$, and denote by 
$\Lambda_0^{\ast}$ the obtained graph, see Figure~\ref{fig}. 

\begin{figure}[t]
\begin{center}
\begin{tikzpicture}[scale=1.5] ---
\fill (0,0) circle (1pt);
\fill (1,0) circle (1pt);
\fill (2,0) circle (1pt);
\fill (3,0) circle (1pt);
\fill (4,0) circle (1pt);
\fill (5,0) circle (1pt);
\fill (6,0) circle (1pt);
\fill (7,0) circle (1pt);
\draw (0,0) node[below]{\footnotesize $1$};
\draw (1,0) node[below]{\footnotesize $b$};
\draw (2,0) node[below]{\footnotesize $bx$};
\draw (3,0) node[below]{\footnotesize $bxb$};
\draw (4,0) node[below]{\footnotesize $(bx)^2$};
\draw (5,0) node[below]{\footnotesize $(bx)^{n-2}b$};
\draw (6,0) node[below]{\footnotesize $(bx)^{n-1}$};
\draw (7,0) node[below]{\footnotesize $(bx)^{n-1}b$};
\draw[->] (0,0) -- (1,0);
\draw[->, dashed] (1,0) -- (2,0); 
\draw[->] (2,0) -- (3,0);
\draw[->, dashed] (3,0) -- (4,0);
\draw[->] (4,0) -- (4.2,0);
\draw[dotted] (4.2,0) -- (4.8,0);
\draw[->] (4.8,0) -- (5,0);
\draw[->, dashed] (5,0) -- (6,0);
\draw[->] (6,0) -- (7,0);
\draw[bend left=-20, dashed] (0,0) to (7,0); 
\draw (0,0) arc[start angle=180, end angle=0, radius=0.5];
\draw (0.5,0.5) node[above]{\footnotesize $C_0$};
\draw (2,0) arc[start angle=180, end angle=0, radius=0.5];
\draw (2.5,0.5) node[above]{\footnotesize $C_1$};
\draw (6,0) arc[start angle=180, end angle=0, radius=0.5];
\draw (6.5,0.5) node[above]{\footnotesize $C_{n-1}$};
\draw (4,0) arc[start angle=180, end angle=135, radius=0.5];
\draw (5,0) arc[start angle=0, end angle=45, radius=0.5];
\draw[->] (0.45,0.5) -- (0.55,0.5);
\draw[->] (2.45,0.5) -- (2.55,0.5);
\draw[->] (6.45,0.5) -- (6.55,0.5);
\draw[->] (3.55,-0.7) -- (3.45,-0.7);
\end{tikzpicture}
\caption{The subgraph $\Lambda_0^{\ast}$ and its orientation $D^{\ast}$; the $x$-edges are drawn as dashed lines.}
\label{fig}
\end{center}
\end{figure}
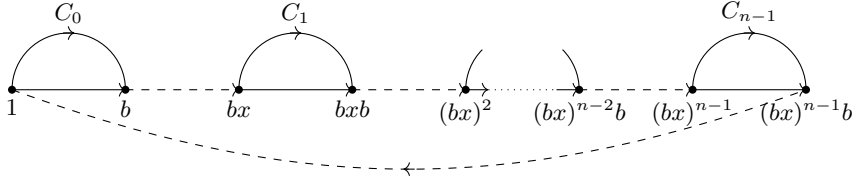

For $g \in G$, the \emph{left translation} $\lambda_g$ is the permutation of $G$, defined as 
$\lambda_g(x)=gx$ ($x \in G$). It is easy to see that $\lambda_g \in \aut(\G)$; 
and therefore, $\lambda_{g}(\Lambda_0^{\ast})$ is a subgraph of $\G$ for every $g \in G$. Finally, define the subgraph $\Lambda$ as 
\[
\Lambda=\bigcup_{i=0}^{p-1}\lambda_{a^i}(\Lambda_0^{\ast}). 
\]

We go through the conditions in Proposition~\ref{new-suffi} case by case.

\medskip
\noindent $\bullet$ (i) \textit{$\Sigma$ is a spanning subgraph with components $\Sigma_1,\ldots,\Sigma_k$, and  
there exists an odd integer $t > 1$ such that $\Sigma_i \cong CL_t$ 
for each $i \in [k]$ (hence $|G|=2tk$).}

\medskip
Clearly, $\Sigma$ is a spanning subgraph. 
Each component of $\Sigma$ is isomorphic to the graph 
$\cay(D,\{a^s,a^{-s},x\}) \cong CL_p$, where $D=\sg{a^p,x}$. 
There are $pm$ components and we list them as $\Sigma_1,\ldots,\Sigma_{pm}$ 
(hence $k=pm$).

\medskip
\noindent $\bullet$ (ii) \textit{$\Sigma \cup \Lambda$ is a parity subgraph of $\G$.}  

\medskip
Every vertex of $\Sigma \cup \Lambda$ has valency $3$ or $5$, because 
$\Sigma$ is a spanning subgraph, which covers all $a^s$- and 
$x$-edges of $\G$, and for every $g \in G$, if $g \in \vertex(\Lambda)$, then both $b$-edges 
of $\G$ incident with $g$ belong to $\edge(\Lambda)$.   

\medskip
\noindent $\bullet$ (iii) \textit{For each $i \in [pm]$, 
$|\edge(\Lambda) \cap \edge(\Sigma_i)| \in \{1,2\}$ and 	
every edge in $\edge(\Lambda) \cap \edge(\Sigma_i)$ is a rung of $\Sigma_i$.}

\medskip
The second part of the statement follows directly from the definitions of 
$\Sigma$ and $\Lambda$. The components of $\Sigma$ are induced by the cosets $gD$, $g \in G$. 
The $x$-edges of $\Lambda$ can be listed as 
\[
\{a^i(bx)^j,a^i(bx)^jx\},\quad 0 \leq i \leq p-1,~0 \leq j \leq n-1.
\]

We claim that these edges are pairwise distinct. For otherwise, 
\[
a^i(bx)^j=a^{i'}(bx)^{j'}x^\varepsilon,\quad (i,j) \neq (i',j'),~ 
\varepsilon \in \{0,1\}.
\]
This implies $Fb^jc^j=Fb^{j'}c^{j'+\varepsilon}$. 
We thus obtain, successively, $j=j'$, $\varepsilon=0$, and hence 
$i=i'$, a contradiction.

We show that the above edges are distributed equally among the $pm$ components of $\Sigma$. 
Let $e$ and $e'$ be two distinct $x$-edges of $\Lambda$, such that 
\[
e=\{a^i(bx)^j,a^i(bx)^jx\}, \quad e'=\{a^{i'}(bx)^{j'},a^{i'}(bx)^{j'}x\}.
\] 
Observe that $e$ and $e'$ belong to the same component of 
$\Sigma$ if and only if $a^i(bx)^j D=a^{i'}(bx)^{j'} D$. Equivalently, 
\begin{equation}\label{eq:D-coset-1}
(bx)^{-j}a^{i'-i}(bx)^j (bx)^{j'-j}~\in~\sg{a^p,ac}.
\end{equation}
If $j'=j$, then $a^{i'-i} \in \sg{a^p}$, implying that $i'=i$ (recall that $0 \leq i,i' \leq p-1$). 
This contradicts the assumption that $e \ne e'$.
We may assume that $j' > j$. Using~\eqref{eq:conj} and that $x=ac$, it is a routine task to deduce that 
\begin{align*}
(bx)^{-j}a^{i'-i}(bx)^j &= a^{(-r)^j(i'-i)}, \\ 
(bx)^{j'-j} &= (bc)^{j'-j}a^{-1+r-r^2+\cdots+(-1)^{j'-j}r^{j'-j-1}}.
\end{align*}
This shows that $b^{j'-j}=1$, hence $j'=j+m$. As $m$ is odd, $(bc)^{j'-j}=c$, and so  
the condition in \eqref{eq:D-coset-1} reduces to 
\[
a^{(-r)^{j+1}(i'-i)+r-r^2+\cdots+r^{m-2}-r^{m-1}} \in \sg{a^p}.
\]
As $0 \leq i, i' \leq p-1$, this shows that $i'$ is uniquely determined by $i$. 
We conclude that $|\edge(\Lambda) \cap \edge(\Sigma_i)|=2$ 
for every $i \in [pm]$. 

\medskip
\noindent $\bullet$ (iv) \textit{$\Lambda$ admits a $\Z_3$-flow $(D,\varphi)$ satisfying the following conditions.
\begin{enumerate}[(a)]
\item $\Null(\varphi) \subseteq \edge(\Sigma)$. 
\item For every $i \in [pm]$, 
$\edge(\Lambda) \cap \edge(\Sigma_i) \not\subset \Null(\varphi)$.
\end{enumerate}}

\medskip
Define the orientation $D^{\ast}$ of $\Lambda_0^{\ast}$ by orienting 
$\Lambda_0$ clockwise, and for every $0 \leq j \leq n-1$, the subgraph  
$C_j \setminus (\edge(\Lambda_0) \cap \edge(C_j))$ also clockwise, see Figure~\ref{fig}.  
Then define the function $\varphi^{\ast} \colon \edge(\Lambda_0^{\ast}) \to \Z_3$, 
\[
\varphi^{\ast}(e)=\begin{cases}
2 & \text{ if $e$ is an a $x$-edge}, \\ 
1 & \text{otherwise}.
\end{cases}
\]
It can be checked quickly that $(D^{\ast},\varphi^{\ast})$ is a nowhere-zero $\Z_3$-flow of 
$\Lambda_0^{\ast}$.  Furthermore, as $\Lambda$ is $p$ disjoint copies of $\Lambda_0^{\ast}$, $(D^{\ast},\varphi^{\ast})$ extends to a nowhere-zero $\Z_3$-flow $(D,\varphi)$ 
of $\Lambda$. Clearly, $(D,\varphi)$ satisfies the conditions in both (a) and (b).  
\end{proof}

In the next lemma we show that $G$ is isomorphic to a 
subgroup of $\AGL_1(p^2)$ provided that 
$F(G)$ is a minimal normal subgroup of $G$. 
The group $\AGL_1(p^2)$ consists of the 
\emph{affine transformations} of the form 
$x \mapsto rx+a$, where $r,a \in \F_{p^2}$, $r \ne 0$. In what follows, we denote this transformation by $[r,a]$. For $[r,a],[s,b] \in \AGL_1(p^2)$, we have 
\[
[r,a][s,b]=[rs,a+rb]. 
\]
Let $\sigma$ denote the 
\emph{Frobenius automorphism} of $\F_p^2$, i.\,e., 
$\sigma(x)=x^p$, where $x \in \F_{p^2}$. 
The group $\AGaL_1(p^2)$ is defined as  
\[
\AGaL_1(p^2)=\sg{\AGL_1(p^2), \sigma}=\AGL_1(p^2)\sg{\sigma}.
\]

\begin{lem}\label{L4}
If $F(G)$ is a minimal normal subgroup of $G$, then 
$G$ is isomorphic to a subgroup of $\AGL_1(p^2)$.
\end{lem}
\begin{proof}
Recall that $G=F(G) \rtimes H$, where $F(G) \cong \Z_p\times \Z_p$ and 
$H \cong \Z_n$. 
Denote by $G/H$ the set of right cosets of $H$ in $G$ and consider the 
the action of $G$ on $G/H$, defined as $(Hx)^g=Hxg$, where $x,g \in G$. Note that $H$ contains no nontrivial normal subgroup of $G$, and therefore, the induced permutation group on $G/H$ is isomorphic to $G$. 
For the coset $\alpha=H$, $G_\alpha=H$. As $F(G)$ is a minimal normal abelian subgroup, it follows that $H$ is maximal in $G$, and therefore, $G$ is primitive on $G/H$.  All conditions in Proposition~\ref{primitive} hold for $G$, 
where $\Omega=G/H$ and $N=F(G)$, and this gives that $G$ is isomorphic to a subgroup of $\AGaL_1(p^2)$. We write 
$G \leq \AGaL_1(p^2)$. 

Let $H=\sg{h}$. We may assume that $h \in \AGaL_1(p^2)$ and $h(0)=0$. 
Then $h=[r,0]$ or $h=[r,0]\sigma$. 
If $h=[r,0]$, then $G \leq \AGL_1(p^2)$, as claimed. 
Suppose that $h=[r,0]\sigma$. We show below that this case cannot occur. 

Consider the $\F_p$-linear map of $\F_p^2$
corresponding to $h$ and denote by $f(x)$ its characteristic polynomial.  
Since $F(G)$ is a minimal normal subgroup in $G$, it follows that 
$f(x)$ is irreducible over $\F_p$. Consequently,  
$r \in \F_{p^2} \setminus \F_p$. For otherwise, $f(x)$ would vanish at 
$r$, contradicting its irreducibility. Thus, $1,r$ form a basis of $\F_p^2$. 
We compute $h(1)=r$, $h(r)=r^{p+1}$, and so that $f(x)=x^2-r^{p+1}=0$. 
This implies that $r^{p+1}$ is a nonsquare 
element of $\F_p$. Therefore, as an element of the multiplicative subgroup 
of $\F_p$, $|r^{p+1}|$ is even. Since $h^2=[r^{p+1},0]$, this shows that 
$n=|h|$ is divisible by $4$. This contradicts our assumption that 
$n$ is square-free. 
\end{proof}

We are ready to derive the final contradiction.  

\begin{lem}\label{L5}
$\G$ admits a nowhere-zero $3$-flow. 
\end{lem}
\begin{proof}
By Lemmas~\ref{L1}-\ref{L3}, we may assume that  
$\G$ is $5$-valent, $G=F(G) \rtimes H$, where $F(G) \cong 
\Z_p\times\Z_p$, and $H \cong \Z_n$. For an ease of notation, let $F=F(G)$. 
We distinguish two cases according to whether $F$ is a minimal 
normal subgroup of $G$.

\medskip 
\noindent{\bf Case~1.} 
{\it $F$ is not a minimal normal subgroup of $G$.}

\smallskip
By Maschke theorem (see, e.\,g.,~\cite[Theorem~1.17.7]{Hbook}), $F$ decomposes as $F=F_1 \times F_2$, 
where both $F_1$ and $F_2$ are normal in $G$ of order $p$. 

If $F_1 \cap X \neq \emptyset$ and $F_2 \cap X \neq \emptyset$, then 
we can use the argument in the second-to-last paragraph in the proof of 
Lemma~\ref{L2} to deduce a nowhere-zero $\Z_3$-flow in $\G$.  

If $F_i \cap X=\emptyset$ and $i \in \{1,2\}$, then 
$\cay(G/F_i,X/F_i)$ admits a nowhere-zero $3$-flow by Proposition~\ref{derived}, and therefore, 
so does $\G$ by Proposition~\ref{quo}. 

\medskip 
\noindent{\bf Case~2.} 
{\it $F$ is a minimal normal subgroup of $G$.}

\medskip
By Lemma~\ref{L4}, $G\leq \AGL_1(p^2)$. Thus, $G=F \rtimes H$, where $H=\sg{h}$, $h=[-r,0]$, as an element of the multiplicative group of $\F_{p^2}$, $|r|=n/2$. 

If $X \cap F=\emptyset$, then $\G$ admits a nowhere-zero $3$-flow by 
Propositions~\ref{quo} and \ref{derived}. 
We assume that $|X \cap F| \geq 2$. If $X \setminus F$ contains only involutions, then 
$G=\sg{X}$ is isomorphic to the generalized dihedral group 
$\mathrm{Dih}(\Z_p\times \Z_p)$, hence $\G$ admits a nowhere-zero $3$-flow by {\cite[Theorem~1]{LL}}. Thus, we may assume that $|X \cap F|=2$ and 
$X \setminus F$ contains an element $y$ of order larger then $2$. 
It follows that $|y|=m$ or $2m$. By Lemma~\ref{condition}, we may also assume that $|y|=m$.  
By Hall theorem, $y$ is conjugate to $[r,0]$. Consequently, $\G \cong \G^\prime:=\cay(G,X')$, 
where 
\[
X'=\big\{~[1,a],[1,-a],[r,0],[r^{-1},0],x~\big\},\quad x=[-1,b],~a,b \in \F_{p^2},~a \ne 0.
\]

\medskip
We copy the proof of Lemma~\ref{L3} and define subgraphs $\Sigma$ and $\Lambda$ of 
$\G^\prime$ and 
show that these satisfy conditions (i)--(iv) in Proposition~\ref{new-suffi}. 

Define  
\[
\Sigma=\cay\big(G,\big\{~[1,a],[1,-a],x~\big\}\big).
\]

Let $L=\sg{[r,0]x}=\sg{[-r,rb]}$. As $|r|=m$ is odd, $r \ne -1$. A direct computation yields 
$[1,-\frac{rb}{r+1}][-r,rb][1,\frac{rb}{r+1}]=[-r,0]$, implying 
that $|L|=|[-r,0]|=n$. Let $\Lambda_0^{\ast}$ be the 
$2n$-cycle, defined as
\[
\Lambda_0 = 1 - [r,0] - [-r,rb] - [-r,rb][r,0] - [-r,rb]^2 -~ \cdots~ - 
[-r,rb]^{n-1} - [-r,rb]^{n-1}[r,0] - 1.
\]
Define $\Lambda_0^*$ to be the graph obtained by gluing the $m$-cycles $C_j$ to 
$\Lambda_0$, where $0 \leq j \leq n-1$, and 
\[
C_j= ([-r,rb])^j - ([-r,rb])^j[r,0] - ([-r,rb])^j[r,0]^2 -~\cdots~ - 
([-r,rb])^j[r,0]^{m-1} - ([-r,rb])^j.
\]

Consider the action of $H$ on the set of subgroups of $F$ of order $p$ by 
conjugation and denote 
by $\bar{H}$ the image of $H$. The the image of 
the subgroup $\sg{[1,a]}$ under the element $h^j \in H$ is 
computed as 
\[
[1,a]^{h^j}=\big\langle~h^{-j}[1,a]h^{j}~\big\rangle=\sg{~[1,(-r)^{-j}a]~}.
\]

Notice that $|\bar{H}|$ is odd because 
$[-1,0]$ is in the kernel of the action and $|H|/2=m$ is odd. 
Using also the fact that $F$ has $p+1$ subgroups of order $p$, we obtain that $\bar{H}$ is intransitive, and therefore, there exists $c \in \F_p^2$ that satisfies
\begin{equation}\label{eq:cap}
\sg{~[1,a]~}~\cap~\bigcup_{i=0}^{n-1}\sg{~[1,(-r)^jc]~}=\{1_G\}.
\end{equation}
Finally, define 
\[
\Lambda=\bigcup_{i \in \F_p}^{p-1}\lambda_{[1,ic]}(\Lambda_0^{\ast}), 
\]

We show that $\Sigma$ and $\Lambda$ satisfy conditions (i)--(iv) in Proposition~\ref{new-suffi}. Let $D=\sg{[1,a],[-1,b]}$. Then, $|D|=2p$, hence $\Sigma$ has $pm$ components. Let $\Sigma_i$, $i \in [pm]$, denote these 
components. It is easy to see that the argument used in the proof of Lemma~\ref{L3} can be directly copied in the case of conditions (i), (ii) and (iv), and therefore, we discuss only condition~(iii).

\medskip
\noindent $\bullet$ (iii) \textit{For each $i \in [pm]$, 
$|\edge(\Lambda) \cap \edge(\Sigma_i)| \in \{1,2\}$ and 	
every edge in $\edge(\Lambda) \cap \edge(\Sigma_i)$ is a rung of $\Sigma_i$.}

\medskip
The second part of the statement follows directly from the definitions of 
$\Sigma$ and $\Lambda$. 
The $x$-edges of $\Lambda$ can be listed as 
\[
\{~[1,ic][-r,rb]^j,[1,ic][-r,rb]^jx~\},\quad 0 \leq i \leq p-1,~
0 \leq j \leq n-1.
\]
We show that these are distributed equally among the $pm$ components of $\Sigma$. 
Let $e$ and $e'$ be two distinct $x$-edges of $\Lambda$, such that 
\begin{align*}
e =&\{~[1,ic][-r,rb]^j,[1,i][-r,rb]^jx~\}, \\ 
e'=& \{~[1,i'c][-r,rb]^{j'},[1,i'c][-r,rb]^{j'}x~\}.
\end{align*}

Then, $e$ and $e'$ belong to the same component of 
$\Sigma$ if and only if $[1,ic][-r,rb]^j D= [1,i'c][-r,rb]^{j'} D$. Equivalently, $z \in D$, where 
\begin{align*}
z &= [-r,br]^{-j}[1,(i'-i)c][-r,br]^j [-r,br]^{j'-j} \\ 
&=\big[~(-r)^{j'-j},(-r)^{-j}(i'-i)c+
    \big(r-r^2+r^3-\cdots+(-r)^{j'-j'}\big)b~\big].
    \end{align*}

If $j'=j$, then $z \in D$ if and only if 
$(-r)^{-j}(i'-i)c \in \sg{a}$. By \eqref{eq:cap}, this forces $i'=i$, 
contradicting that $e \neq e'$. 

Let $j' > j$. Then $z \in D$ if and only if $j'=j+m$ and 
\[
[1,(-r)^{-j}(i'-i)c] \in \big\langle~[1,a]~\big\rangle+
\big[~1,\big(1-r+r^2-r^3+\cdots+(-r)^{j'-j'}\big)b~\big].
\]
By \eqref{eq:cap}, we obtain that $i'$ is uniquely determined by $i$. 
We conclude that $|\edge(\Lambda) \cap \edge(\Sigma_i)|=2$ 
for every $i \in [pm]$. 
\end{proof}

\end{document}